\documentclass[12pt]{amsart}

\usepackage{amsmath,amssymb, verbatim,color}
\usepackage{mathrsfs}
\usepackage{enumerate}
\usepackage[active]{srcltx}

\newtheorem{theorem}{Theorem}[section]
\newtheorem{corollary}[theorem]{Corollary}
\newtheorem{lemma}[theorem]{Lemma}

\newtheorem{example}[theorem]{Example}
\newtheorem{remark}{Remark}

\date{\today}

\title{On fractional derivatives and primitives of periodic functions}

\author[Area]{I. Area}
\address[Area]{Departamento de Matem\'{a}tica Aplicada II, E.E. Telecomunicaci\'{o}n, 
Universidade de Vigo, 36310-Vigo, Spain.}
\email[Area]{area@uvigo.es}

\author[Losada]{J. Losada}
\address[Losada]{Facultade de Matem\'{a}ticas, Universidade de Santiago de
Compostela, 15782-Santiago de Compostela, Spain}
\email[Losada]{jorge.losada@rai.usc.es}

\author[Nieto]{J. J. Nieto}
\address[Nieto]{Facultade de Matem\'{a}ticas, Universidade de Santiago de
Compostela, 15782-Santiago de Compostela, Spain, and Faculty of Science, King Abdulaziz University, P.O. Box 80203, 21589, Jeddah, Saudi Arabia. Corresponding author.}
\email[Nieto]{juanjose.nieto.roig@usc.es}


\begin{document}

\begin{abstract}
In this paper we prove that the fractional derivative or the fractional primitive of a $T$-periodic function cannot be a $\tilde{T}$-periodic function, for any period $\tilde{T}$, with the exception of the zero function.
\end{abstract}
\maketitle

\section{Introduction}

Periodic functions \cite[Ch. 3, pp. 58-92]{knopp1996theory} play a central role in mathematics since the seminal works of Fourier \cite{MR0179403,zbMATH03112145}. Nowadays periodic functions appear in applications ranging from electromagnetic radiation to blood flow, and of course in control theory in linear time-varying systems driven by periodic input signals \cite{zbMATH05316781}. Linear time-varying systems driven by periodic input signals are ubiquitous in control systems, from natural sciences to engineering, economics, physics and the life science \cite{zbMATH05316781,MR2942338}. Periodic functions also appear in automotive engine applications \cite{MR2306745}, optimal periodic scheduling of sensor networks \cite{MR3029117,MR3080465}, or cyclic gene regulatory networks \cite{MR2889214}, to give some applications.

It is an obvious fact that the classical derivative, if it exists, of a periodic function is also a periodic function of the same period. Also the primitive of a periodic function may be periodic (for example,  $\cos t$  as primitive of  $\sin t$).

The idea of integral or derivatives of noninteger order goes back to Riemann and Liouville \cite{zbMATH03112145,MR2863974}. Probably the first application of fractional calculus was made by Abel in the solution of the integral equation that arises in the formulation of the tautochrone problem \cite{Lagrange1765}. Fractional calculus appears in many different contexts as speech signals, cardiac tissue electrode interface, theory of viscoelasticity, or fluid mechanics. The asymptotic stability of positive fractional-order nonlinear systems has been proved in \cite{MR3163802} by using the Liapunov function. We do not intend to give a full list of applications but to show the wide range of them.

In this paper we prove that periodicity is not transferred by fractional integral or derivative, with the exception of the zero function. Although this property seems to be known \cite{MR2863974,MR2877170,MR2879511}, in Section \ref{S:3} we give a different proof by using the Laplace transform. Our approach relies on the classical concepts of fractional calculus and elementary analysis. Moreover, by using a similar argument as in \cite{MR2971825}, in Section \ref{S:4} we prove that the fractional derivative or primitive of a $T$-periodic function cannot be $\tilde{T}$-periodic for any period $\tilde{T}$. A particular but nontrivial example is explicitly given. Finally, as a consequence we show in Section \ref{S:5} that an autonomous fractional differential equation cannot have periodic solutions with the exception of constant functions.

\section{Preliminares}

Let $T>0$. If $f:{\mathbb{R}} \to {\mathbb{R}}$ is $T$ periodic and $f \in {\mathcal{C}}^{1}({\mathbb{R}})$, then the derivative $f'$ is also $T$-periodic. However, the primitive of $f$
\begin{equation}\label{eq:21}
F(t)=\int_{0}^{t} f(s)ds
\end{equation}
is not, in general, $T$-periodic. Just take $f(t)=1$ so that $F(t)=t$ is not $T'$-periodic for any $T'>0$. The necessary and sufficient condition for $F$ to be $T$-periodic is that
\begin{equation}\label{eq:22}
\int_{0}^{T} f(s)ds=0.
\end{equation}

The purpose of this note is to show that the fractional derivative or the fractional primitive of a $T$-periodic function cannot be $T$-periodic function with the exception, of course, of the zero function. We use the notation
\[
F=I^{1}f, \qquad f'=D^{1}f
\]
and note that
\[
D^{1} (I^{1}f) (t)=D^{1}F(t)=f(t)
\]
but
\[
I^{1}(D^{1}f) (t)=f(t)-f(0)
\]
and $I^{1}(D^{1}f)$ does not coincide with $f$ unless $f(0)=0$.

We recall some elements of fractional calculus. Let $\alpha \in (0,1)$ and $f:{\mathbb{R}} \to {\mathbb{R}}$. We point out that $f$ is not necessarily continuous. The fractional integral of $f$ of order $\alpha$ is defined by \cite{MR2218073}
\begin{equation}\label{eq:23}
I^{\alpha} f(t)=\frac{1}{\Gamma(\alpha)} \int_{0}^{t} (t-s)^{\alpha-1} f(s)ds,
\end{equation}
provided the right-hand side is defined for a.e. $t \in {\mathbb{R}}$. If, for example, $f \in {\mathcal{L}}^{1}({\mathbb{R}})$, then the fractional integral (\ref{eq:23}) is well defined and $I^{\alpha}f \in {\mathcal{L}}^{1}(0,T)$, for any $T>0$. Moreover, the fractional operator
\[
I^{\alpha}: {\mathcal{L}}^{1}(0,T) \to {\mathcal{L}}^{1}(0,T)
\]
is linear and bounded.

The fractional Riemann-Liouville derivative of order $\alpha$ of $f$ is defined as \cite{MR2218073,MR1658022}
\[
D^{\alpha}f(t)=D^{1} I^{1-\alpha} f(t) = \frac{1}{\Gamma(1-\alpha)} \frac{d}{dt} \int_{0}^{t} (t-s)^{\alpha} f(s)ds.
\]
This is well defined if, for example, $f \in {\mathcal{L}}^{1}_{\text{loc}}({\mathbb{R}})$.

There are many more fractional derivatives. We are not giving a complete list, but recall the Caputo derivative \cite{MR2218073,MR1658022}
\[
{}^{c}D^{\alpha}f(t)=I^{1-\alpha}D^{1}f(t)=\frac{1}{\Gamma(1-\alpha)} \int_{0}^{t} (t-s)^{-\alpha} f'(s)ds,
\]
which is well defined, for example, for absolutely continuous functions.

As in the integer case we have
\[
D^{\alpha} (I^{\alpha}f) (t)=f(t), \qquad \,{}^{c}D^{\alpha}( I^{\alpha}f ) (t)=f(t)
\]
but $I^{\alpha}(D^{\alpha}f)$ or $I^{\alpha}({}^{c}D^{\alpha}f)$ are not, in general, equal to $f$. Indeed
\[
I^{\alpha}({}^{c}D^{\alpha}f)(t)=f(t)-f(0),
\]
and (see \cite[(2.113), p. 71]{MR1658022})
\[
I^{\alpha}(D^{\alpha}f)(t)=f(t)-\frac{D^{\alpha-1}f(0)}{\Gamma(\alpha)}t^{\alpha-1}.
\]
Also \cite[(2.4.4), p. 91]{MR2218073}
\[
{}^{c}D^{\alpha}f(t)=D^{\alpha}(f(t)-f(0)).
\]

\section{The fractional derivative or primitive of a $T$-periodic function cannot be $T$-periodic}\label{S:3}
We prove the following result in Section 3.1 below:
\begin{theorem}\label{t:21}
Let $f:{\mathbb{R}} \to {\mathbb{R}}$ be a nonzero $T$-periodic function with $f \in {\mathcal{L}}^{1}_{\text{loc}}({\mathbb{R}})$. Then $I^{\alpha}f$ cannot be $T$-periodic for any $\alpha \in (0,1)$.
\end{theorem}

\begin{corollary}\label{c:22}
Let $f:{\mathbb{R}} \to {\mathbb{R}}$ be a nonzero $T$-periodic function such that $f \in {\mathcal{L}}^{1}_{\text{loc}}({\mathbb{R}})$. Then, the Caputo derivative ${}^{c}D^{\alpha}f$ cannot be $T$-periodic for any $\alpha \in (0,1)$. The same result holds for the fractional derivative $D^{\alpha}f$.
\end{corollary}
\begin{proof}
Suppose that ${}^{c}D^{\alpha}f$ is $T$-periodic. Then, by Theorem \ref{t:21}, $I^{\alpha} ({}^{c}D^{\alpha}f)$ cannot be $T$-periodic. However
\[
I^{\alpha} ({}^{c}D^{\alpha}f)(t)=f(t)-f(0)
\]
is $T$-periodic. In relation to the fractional Riemann-Liouville derivative, suppose that $D^{\alpha}f$ is $T$-periodic and consider the function $\hat{f}=f-f(0)$ which is also $T$-periodic. Then
\[
{}^{c}D^{\alpha}\hat{f}=D^{\alpha} \hat{f}
\]
cannot be $T$-periodic.
\end{proof}

\subsection{Proof of Theorem \ref{t:21}}

Let $\alpha \in (0,1)$ and $T>0$. By reduction to the absurd, in this section we suppose that $I^{\alpha}f$ is $T$-periodic. Then
\[
I^{\alpha}f(0)=0=I^{\alpha}f(T),
\]
that is,
\begin{equation}\label{eq:311}
\int_{0}^{T} (T-s)^{\alpha-1}f(s)ds=0.
\end{equation}

\begin{lemma}\label{l:31}
Assume $f \in {\mathcal{L}}^{1}_{loc}({\mathbb{R}})$ is $T$-periodic. If $I^{\alpha}f$ is also $T$-periodic, then
\begin{equation}\label{eq:31}
\int_{0}^{T} (nT-s)^{\alpha-1} f(s)ds=0, \qquad  (n \in {\mathbb{N}}:=\{1,2,3,\dots\}).
\end{equation}
\end{lemma}
\begin{proof}
For $n=1$ the latter equality reduces to (\ref{eq:311}). For $n=2$,
\begin{multline*}
0=I^{\alpha} f(2T) = \frac{1}{\Gamma(\alpha)} \int_{0}^{2T} (2T-s)^{\alpha-1}f(s)ds \\
=\frac{1}{\Gamma(\alpha)} \int_{0}^{T} (2T-s)^{\alpha-1} f(s)ds + \frac{1}{\Gamma(\alpha)} \int_{T}^{2T} (2T-s)^{\alpha-1} f(s)ds \\
=\frac{1}{\Gamma(\alpha)} \int_{0}^{T} (2T-s)^{\alpha-1} f(s)ds + \frac{1}{\Gamma(\alpha)} \int_{0}^{T} (T-r)^{\alpha-1} f(r+T) dr \\
=\frac{1}{\Gamma(\alpha)} \int_{0}^{T} (2T-s)^{\alpha-1} f(s)ds + \frac{1}{\Gamma(\alpha)} \int_{0}^{T} (T-r)^{\alpha-1} f(r) dr 
\\ =\frac{1}{\Gamma(\alpha)} \int_{0}^{T} (2T-s)^{\alpha-1}f(s)ds.
\end{multline*}
The proof follows by induction on $n$. Assume that (\ref{eq:31}) is valid for some $n \in {\mathbb{N}}$. Then
\[
\int_{0}^{(n+1)T} ((n+1)T-s)^{\alpha-1}f(s)ds=\sum_{j=0}^{n} \int_{jT}^{(j+1)T} ((n+1)T-s)^{\alpha-1}f(s)ds,
\]
and, by periodicity,
\[
\int_{0}^{(n+1)T} ((n+1)T-s)^{\alpha-1}f(s)ds=I^{\alpha} f((n+1)T)=0.
\]
Moreover, for $j=1,2,\dots,n$
\[
\sum_{j=1}^{n} \int_{jT}^{(j+1)T} ((n+1)T-s)^{\alpha-1}f(s)\,ds = \sum_{j=1}^{n}\int_{0}^{T} ((n+1-j)T-r)^{\alpha-1} f(r) \,dr=0
\]
by hypothesis of induction since $1 \leq n+1-j \leq n$. Hence
\[
0=\sum_{j=0}^{n} \int_{jT}^{(j+1)T} ((n+1)T-s)^{\alpha-1}f(s)ds= \int_{0}^{T} ((n+1)T-s)^{\alpha-1}f(s)ds.
\] 
\end{proof}

\begin{lemma}\label{l:2}
Under the hypothesis of {\rm{Lemma \ref{l:31}}},
\begin{equation}\label{eq:32}
\int_{0}^{T} f(s)ds=0.
\end{equation}
\end{lemma}
\begin{proof}
Let $f^{+}$ and $f^{-}$ be the positive and negative parts of $f$, 
\[
f^{+}(x)=\max(f(x),0), \quad f^{-}(x)=-\min(f(x),0), \quad  f=f^{+}-f^{-}.
\]
Equation (\ref{eq:31}) implies that
\[
\int_{0}^{T} (nT-s)^{\alpha-1} f^{+}(s) ds = \int_{0}^{T} (nT-s)^{\alpha-1} f^{-}(s)ds.
\]
If $\int_{0}^{T}f^{+}(s)ds=0$ or $\int_{0}^{T}f^{-}(s)ds=0$, then from (\ref{eq:31}) we get $f=0$. We consider the case
\[
\int_{0}^{T} f^{+}(s) ds > \int_{0}^{T} f^{-}(s)ds >0.
\]
For $n$ large
\[
\left( \frac{nT}{(n-1)T} \right)^{\alpha-1} > \frac{\int_{0}^{T} f^{-}(s)ds}{\int_{0}^{T} f^{+}(s)ds}
\]
or equivalently
\[
(nT)^{\alpha-1} \int_{0}^{T} f^{+}(s)ds > ((n-1)T)^{\alpha-1} \int_{0}^{T} f^{-}(s)ds.
\]
Hence
\begin{multline*}
0=\int_{0}^{T}(nT-s)^{\alpha-1} f(s)ds 
\geq (nT)^{\alpha-1} \int_{0}^{T} f^{+}(s)ds - ((n-1)T)^{\alpha-1} \int_{0}^{T} f^{-}(s)ds >0,
\end{multline*}
which is a contradiction.

The case
\[
\int_{0}^{T} f^{-}(s)ds > \int_{0}^{T} f^{+}(s)ds>0
\]
is analogous.

Therefore
\[
\int_{0}^{T} f^{-}(s)ds = \int_{0}^{T} f^{+}(s)ds>0,
\]
and
\[
\int_{0}^{T} f(s)ds=0.
\]
\end{proof}

\begin{lemma}\label{l:3}
Under the hypothesis of {\rm{Lemma \ref{l:31}}},
\begin{equation}\label{eq:33}
\int_{0}^{T} (T+\delta-s)^{\alpha-1} f(s)ds=0, \qquad \forall \delta \in [0,T].
\end{equation}
\end{lemma}
\begin{proof}
If $\delta=0$ and $\delta=T$, the equation reduces to (\ref{eq:311}) and (\ref{eq:31}), respectively. Let $0<\delta<T$.
\begin{multline*}
I^{\alpha} f(T+\delta) = \frac{1}{\Gamma(\alpha)} \int_{0}^{T+\delta} (T+\delta-s)^{\alpha-1} f(s) ds \\
=\frac{1}{\Gamma(\alpha)} \int_{0}^{T} (T+\delta-s)^{\alpha-1} f(s) ds +
\frac{1}{\Gamma(\alpha)} \int_{T}^{T+\delta} (T+\delta-s)^{\alpha-1} f(s) ds \\
=\frac{1}{\Gamma(\alpha)} \int_{0}^{T} (T+\delta-s)^{\alpha-1} f(s) ds +
\frac{1}{\Gamma(\alpha)} \int_{0}^{\delta} (\delta-r)^{\alpha-1} f(r+T) dr\\
=\frac{1}{\Gamma(\alpha)} \int_{0}^{T} (T+\delta-s)^{\alpha-1} f(s)ds + I^{\alpha}f(\delta). 
\end{multline*}
By using the periodicity of $I^{\alpha}f$ we get (\ref{eq:33}).
\end{proof}

\begin{lemma}\label{l:4}
Under the hypothesis of {\rm{Lemma \ref{l:31}}},
\begin{equation}\label{eq:34}
\int_{0}^{T} (T+t-s)^{\alpha-1} f(s)ds=0, \qquad \forall t \in {\mathbb{R}}.
\end{equation}
\end{lemma}
\begin{proof}
For $t \in [0,T]$ or $t=nT$, $n=1,2,\dots$, relation (\ref{eq:34}) is true. Let $t=nT+\delta$, so that $T+t=(n+1)T+\delta$. Then,
\[
I^{\alpha}f(\delta)=I^{\alpha} f(T+t) = \frac{1}{\Gamma(\alpha)} \int_{0}^{(n+1)T+\delta} ((n+1)T+\delta-s)^{\alpha-1} f(s)ds.
\]
Now, using the additive property of the integral, we have
\begin{multline*}
\frac{1}{\Gamma(\alpha)} \int_{0}^{(n+1)T+\delta} ((n+1)T+\delta-s)^{\alpha-1} f(s)ds \\
=\frac{1}{\Gamma(\alpha)} \sum_{j=0}^{n} \int_{jT}^{(j+1)T} ((n+1)T+\delta-s)^{\alpha-1} f(s)ds \\ +
\frac{1}{\Gamma(\alpha)} \int_{(n+1)T}^{(n+1)T+\delta} ((n+1)T+\delta-s)^{\alpha-1} f(s)ds.
\end{multline*}
Let us compute separately the integrals in the right hand side. In all the integrals depending on $j$, we use the (linear) change of variable $r=s-jT$ and rename $t'=(n-j)T+\delta$ to obtain
\begin{multline*}
\int_{jT}^{(j+1)T} (nT+T+\delta-s)^{\alpha-1} f(s)ds 
= \int_{0}^{T} (T+(n-j)T+\delta-r)^{\alpha-1} f(r+jT)dr \\
=\int_{0}^{T} (T+t'-s)^{\alpha-1} f(s)ds.
\end{multline*}
For the last integral we use the (linear) change of variable $r=s-(n+1)T$ to get
\[
\int_{(n+1)T}^{(n+1)T+\delta} ((n+1)T+\delta-s)^{\alpha-1} f(s)ds=
\int_{0}^{\delta} (\delta-r)^{\alpha-1} f(r+(n+1)T) dr = I^{\alpha}f(\delta).
\]
By induction on $n$, as in Lemma \ref{l:31}, the proof follows.
\end{proof}

\begin{lemma}
Let $f$ be a continuous and $T$-periodic function, $T>0$. Let $0<\alpha<1$ be fixed. Assuming that
\[
\int_{0}^{T} (T-s+t)^{\alpha-1} f(s)=0, \quad \forall t \in {\mathbb{R}}, \qquad \int_{0}^{T} f(s)ds=0,
\]
then $f \equiv 0$.
\end{lemma}
\begin{proof}
Since $\int_{0}^{T} f(s)ds=0$ then $0=\int_{0}^{T} f(s) ds=\int_{0}^{T} (f^{+}(s)-f^{-}(s)) ds$ and therefore
we can define $c=\int_{0}^{T} f^{+}(s)ds = \int_{0}^{T} f^{-s}(s)ds>0$. If $c=0$ then $f=0$.

Let us define
\[
\phi(t)=\int_{0}^{T} (T-s+t)^{\alpha-1} f(s)ds.
\]
From the hypothesis we have that $\phi(t)=0$ at any $t \in {\mathbb{R}}$. Therefore its integral is also zero. Let us integrate with respect to $t$ from $a$ to $b$ for $0 \leq a \leq b \leq T$:
\begin{multline*}
0=\int_{a}^{b} \phi(t) dt= \int_{a}^{b} \left( \int_{0}^{T} (T-s+t)^{\alpha-1} f(s)ds \right) dt \\
=\int_{0}^{T} \left( \int_{a}^{b} (T-s+t)^{\alpha-1} dt \right) f(s) ds \\
=\int_{0}^{T} \left( \frac{(b-s+T)^{\alpha }-(a-s+T)^{\alpha }}{\alpha } \right) f(s) ds
\end{multline*}
where we have assumed $0 \leq a < b$, $s < T$. Thus,
\[
\int_{0}^{T} \left[ (b-s+T)^{\alpha}-(a-s+T)^{\alpha} \right] f(s) ds=0
\]
which implies that
\[
\psi(t)=\int_{0}^{T} (T-s+t)^{\alpha} f(s)ds
\]
is a constant function.

Moreover, since
\[
t^{\alpha}c-(T+t)^{\alpha}c \leq \int_{0}^{T} (T-s+t)^{\alpha} f(s)ds \leq (T+t)^{\alpha}c-t^{\alpha}c,
\]
where
\[
c=\int_{0}^{T} f^{+}(s)ds=\int_{0}^{T} f^{-}(s)ds
\]
in view of (\ref{eq:32}), and
\[
\lim_{t \to +\infty} ((T+t)^\alpha-t^\alpha)=0,
\]
we have that
\[
\int_{0}^{T} (T-s+t)^{\alpha} f(s)ds=0, \quad \forall t \in {\mathbb{R}}.
\]

Let
\[
\tilde{f}=f\cdot \chi_{[0,T]},  \qquad \tilde{f}(t)=\begin{cases} f(t), & t \in [0,T] \\ 0,& t>T.\end{cases}
\]
If we define
\[
\varphi(t)=(T+t)^{\alpha}
\]
then the convolution of $\varphi$ and $\tilde{f}$ is given by
\[
(\varphi * \tilde{f})=\int_{0}^{+\infty} \varphi(t-s) \tilde{f}(s) ds = \int_{0}^{T} (T+t-s)^{\alpha} f(s)ds=0.
\]
Therefore, if we apply the Laplace transform \cite[Chapter 17]{MR0197789} to the above equality it yields
\[
{\mathcal{L}} [\varphi * \tilde{f}] = {\mathcal{L}} [\varphi]  {\mathcal{L}} [\tilde{f})]= {\mathcal{L}} [0]=0.
\]
Since
\[
{\mathcal{L}} [\varphi]=s^{-\alpha -1} e^{s T} \Gamma (\alpha +1,s T),
\]
where $\Gamma(a,z)$ denotes the incomplete gamma function \cite[Section 6.5]{0171.38503}, then ${\mathcal{L}} [\varphi] \neq 0$ which implies that ${\mathcal{L}} [\tilde{f}]=0$ and therefore $\tilde{f}=0$, i.e. $f=0$ on $[0,T]$.
\end{proof}

\section{The fractional derivative or primitive of a $T$-periodic function cannot be $\tilde{T}$-periodic for any period $\tilde{T}$}\label{S:4}

Let $f$ be a $T$-periodic function and consider $u$ such that
\[
{}^{c}D^{\alpha}u=f(t), \qquad 0<\alpha<1.
\]
Then,
\[
u(t)=u(0)+I^{\alpha} f(t), 
\]
and therefore
\[
{\mathcal{L}}[u(t)] = {\mathcal{L}} u_{0} + {\mathcal{L}} [I^{\alpha} f(t)].
\]
Let us assume that $u$ is a $\tilde{T}$-periodic function. Then, by using some basic properties of the Laplace transform it yields
\[
\frac{\int_{0}^{\tilde{T}} u(t) \exp(-\lambda t) dt}{1-\exp(-\lambda \tilde{T})} = \frac{u_{0}}{\lambda} + \frac{1}{\lambda^{\alpha}} \frac{\int_{0}^{T} f(t) \exp(-\lambda t)dt}{1-\exp(-\lambda T)}.
\]
Therefore
\begin{multline*}
\lambda (1-\exp(-\lambda T)) \int_{0}^{\tilde{T}} u(t) \exp(-\lambda t) dt \\
=u_{0} (1-\exp(-\lambda T)) (1-\exp(-\lambda \tilde{T}))
\\ +\lambda^{1-\alpha} (1-\exp(-\lambda \tilde{T})) \int_{0}^{T} f(t) \exp(-\lambda t) dt.
\end{multline*}
Let us consider $v=u-u_{0}$ so that $v$ is also $\tilde{T}$-periodic and $v(0)=0$. The above equality becomes
\begin{multline*}
\lambda (1-\exp(-\lambda T)) \int_{0}^{\tilde{T}} v(t) \exp(-\lambda t) dt 
 =\lambda^{1-\alpha} (1-\exp(-\lambda \tilde{T})) \int_{0}^{T} f(t) \exp(-\lambda t) dt,
\end{multline*}
or equivalently
\[
\lambda^{\alpha} \frac{(1-\exp(-\lambda T))}{(1-\exp(-\lambda \tilde{T}))} \int_{0}^{\tilde{T}} v(t) \exp(-\lambda t) dt  = \int_{0}^{T} f(t) \exp(-\lambda t) dt.
\]
Thus,
\[
\frac{(1-\exp(-\lambda T))}{(1-\exp(-\lambda \tilde{T}))} \sum_{i=0}^{\infty} (-1)^{i} \frac{\lambda^{\alpha + i}}{i!} \int_{0}^{\tilde{T}} v(t) t^{i} dt = \sum_{i=0}^{\infty} (-1)^{i} \frac{\lambda^{i}}{i!} \int_{0}^{T} f(t) t^{i}dt.
\]
Since
\[
\lim_{\lambda \to 0^{+}} \frac{(1-\exp(-\lambda T))}{(1-\exp(-\lambda \tilde{T}))}  = \frac{T}{\tilde{T}}, \qquad \lim_{\lambda \to 0^{+}} \lambda^{\alpha + i}=0,
\]
by using $0<\alpha<1$ and $i \geq 0$, the limit as $\lambda \to 0^{+}$ of the left hand side is zero, which implies
\[
\int_{0}^{T} f(t)dt=0.
\]
Then,
\begin{multline*}
\frac{(1-\exp(-\lambda T))}{(1-\exp(-\lambda \tilde{T}))} \sum_{i=0}^{\infty} (-1)^{i} \frac{\lambda^{ i}}{i!} \int_{0}^{\tilde{T}} v(t) t^{i} dt =\lambda^{-\alpha} \sum_{i=1}^{\infty} (-1)^{i} \frac{\lambda^{i}}{i!} \int_{0}^{T} f(t) t^{i}dt\\
=\lambda^{1-\alpha} \sum_{i=0}^{\infty} (-1)^{i+1} \frac{\lambda^{i}}{(i+1)!} \int_{0}^{T} f(t) t^{i+1}dt.
\end{multline*}
If we consider $\lambda \to 0^{+}$ in the latter expression we get
\[
\frac{T}{\tilde{T}} \int_{0}^{\tilde{T}} v(t)dt=0,
\]
and therefore
\[
\int_{0}^{\tilde{T}} v(t)dt=0.
\]
By induction, we obtain that
\[
\int_{0}^{T} f(t) t^{i}dt=0, \qquad \int_{0}^{\tilde{T}} v(t) t^{i}dt=0, \qquad i=0,1,2,\dots.
\]
Therefore, $f=u=0$ and there are no nonzero $\tilde{T}$-periodic $L^{\infty}$-solutions of the problem.

\begin{example}
Let $f(t)=\sin(t)$ and $0 < \alpha < 1$. The Caputo-fractional derivative of $f(t)$ is given by 
\[
{}^{c}D^{\alpha} f(t) = \frac{t^{1-\alpha }}{\Gamma (2-\alpha )} \, \, _1F_2\left(1;\frac{3-\alpha}{2},1-\frac{\alpha }{2};-\frac{ t^2}{4}\right),
\]
where the hypergeometric series $\,_{1}F_{2}(a;b,c;d)$  is defined as \cite[Chapter 15]{NIST:DLMF,Olver:2010:NHMF}
\[
\,_{1}F_{2}(a;b,c;d) = \sum_{j=0}^{\infty} \frac{(a)_{j}}{j! (b)_{j}\,(c)_{j}} d^{j},
\]
and the Pochhammer symbol $(A)_{j}=A(A+1)\cdots(A+j-1)$, with $(A)_{0}=1$.

Since
\[
\frac{{}^{c}D^{\alpha} f(\pi)}{{}^{c}D^{\alpha} f(\pi+\tilde{T})}=\frac{\pi ^{1-\alpha } (\tilde{T}+\pi )^{\alpha -1} \, _1F_2\left(1;1-\frac{\alpha
   }{2},\frac{3}{2}-\frac{\alpha }{2};-\frac{\pi ^2}{4}\right)}{\, _1F_2\left(1;1-\frac{\alpha
   }{2},\frac{3}{2}-\frac{\alpha }{2};-\frac{1}{4} (\tilde{T}+\pi )^2\right)},
\]
and
\[
\frac{{}^{c}D^{\alpha} f(\pi/2)}{{}^{c}D^{\alpha} f(\pi/2+\tilde{T})}=
\frac{\left(\frac{2 T}{\pi }+1\right)^{\alpha -1} \, _1F_2\left(1;1-\frac{\alpha }{2},\frac{3}{2}-\frac{\alpha }{2};-\frac{\pi ^2}{16}\right)}{\,
   _1F_2\left(1;1-\frac{\alpha }{2},\frac{3}{2}-\frac{\alpha }{2};-\frac{1}{16} (2 T+\pi )^2\right)},
\]
we have that ${}^{c}D^{\alpha} f(t)$ is not a $\tilde{T}$-periodic function for any positive $\tilde{T}$ and $\alpha \in (0,1)$.  Plotting both functions $\sin(t)$ and ${}^{c}D^{\alpha}\sin(t)$, this last function seems to be periodic but it is not according to our results.

Notice that Kaslik and Sivasundaram \cite{MR2863974} gave the following alternate representation
\[
\,^{c}D^{\alpha} \sin(t)=\frac{1}{2} t^{1-\alpha} \left[ E_{1,2-\alpha}(it)+E_{1,2-\alpha}(-it) \right],
\]
in terms of the two-parameter Mittag-Leffler function \cite[Chapter 10]{NIST:DLMF,Olver:2010:NHMF}
\[
E_{\alpha,\beta}(z)=\sum_{k=0}^{\infty} \frac{z^{k}}{\Gamma(\alpha k + \beta)}.
\]
\end{example}

\section{Periodic solutions of fractional differential equations}\label{S:5}

In this section we show how Theorem \ref{t:21}, can be used to give a nonexistence result of periodic solutions for fractional differential equations. 

Consider the first order ordinary differential equation
\begin{equation}\label{eq:41}
D^{1}u(t)=\varphi(u(t)), \qquad t \in {\mathbb{R}},
\end{equation}
where $\varphi: {\mathbb{R}} \to {\mathbb{R}}$ is continuous. An important question is the existence of periodic solutions \cite{MR2339780,MR2761514,MR1165534}.

If $u:{\mathbb{R}} \to {\mathbb{R}}$ is a $T$-periodic solution of (\ref{eq:41}) then obviously
\begin{equation}\label{eq:42}
u(0)=u(T).
\end{equation}
One can find $T$-periodic solutions of (\ref{eq:41}) by solving the equation only on the interval $[0,T]$ and then checking the values $u(0)$ and $u(T)$. If (\ref{eq:42}) holds, then extending by $T$-periodicity the function $u(t)$, $t \in [0,T]$, to ${\mathbb{R}}$ we have a $T$-periodic solution of (\ref{eq:41}).

However, this is not possible for a fractional differential equation. Consider, for $\alpha \in (0,1)$, the equation
\begin{equation}\label{eq:43}
{}^{c}D^{\alpha}u(t)=\varphi(u(t)), \qquad t \in {\mathbb{R}}.
\end{equation}
If $u$ is a solution of (\ref{eq:43}), let $f(t)=\varphi(u(t))$. Then
\begin{equation}\label{eq:44}
u(t)=u(0)+I^{\alpha}f(t).
\end{equation}
In the case that $u$ is a $T$-periodic solution of (\ref{eq:43}) we have that $f$ is also $T$-periodic. According to Theorem \ref{t:21}, $I^{\alpha}f$ cannot be $T$-periodic unless it is the zero function and we have the following relevant result:
\begin{theorem}\label{t:41}
The fractional equation {\rm{(\ref{eq:43})}} cannot have periodic solutions with the exception of constant functions $u(t)=u_{0}$, $t\in {\mathbb{R}}$, with $\varphi(u_{0})=0$.
\end{theorem}

\begin{remark}\label{r:42}
It is possible to consider the periodic boundary value problem
\begin{equation}
\begin{cases} {}^{c}D^{\alpha}u(t)=\varphi(u(t)), & t \in [0,T], \\
u(0)=u(T), &
\end{cases}
\end{equation}
as in, for example, \cite{MR2525590}, but one cannot extend the solution of that periodic boundary value problem on $[0,T]$ to a $T$-periodic solution on ${\mathbb{R}}$ (unless $u$ is a constant function, as indicated in {\rm{Theorem \ref{t:41})}}.
\end{remark}

\begin{remark}\label{r:43}
The same applies to the Riemann-Liouville fractional differential equation
\[
D^{\alpha} u(t) = \varphi(u(t)), \qquad t \in {\mathbb{R}},
\]
taking into account that
\[
\lim_{t \to 0^{+}} t^{1-\alpha} u(t) = \frac{D^{\alpha-1} u(0)}{\Gamma(\alpha)}.
\]
\end{remark}

\begin{example}
Considering the fractional equation
\begin{equation}\label{eq:46}
{}^{c}D^{\alpha}u(t)=\psi(t,u(t)), \qquad t \in {\mathbb{R}},
\end{equation}
with $\psi:{\mathbb{R}}^{2} \to {\mathbb{R}}$ defined by
\[
\psi(t,u)=u+\frac{t^{1-\alpha }}{\Gamma (2-\alpha )} \, \, _1F_2\left(1;\frac{3-\alpha}{2},1-\frac{\alpha }{2};-\frac{ t^2}{4}\right)-\sin(t),
\]
we have that $u(t)=\sin(t)$ is a $2\pi$-periodic solution of {\rm{(\ref{eq:46})}}. This shows that the result of {\rm{Theorem \ref{t:41}}} is not valid for a non autonomous fractional differential equation as {\rm{(\ref{eq:46})}}.
\end{example}

\section{Conclusion}
By using the classical concepts of fractional calculus and elementary analysis, we have proved that periodicity is not transferred by fractional integral or derivative, with the exception of the zero function. We have also proved that the fractional derivative or primitive of a $T$-periodic function cannot be $\tilde{T}$-periodic for any period $\tilde{T}$. As a consequence we have showed that an autonomous fractional differential equation cannot have periodic solutions with the exception of constant functions.

\section*{Acknowledgements}

The referees and editor deserve special thanks for careful reading and many useful comments and suggestions which have improved the manuscript. The work of I. Area has been partially supported by the Ministerio de Econom\'{\i}a y Competi\-tividad of Spain under grant MTM2012--38794--C02--01, co-financed by the European Community fund FEDER. J.J. Nieto also acknowledges partial financial support by the Ministerio de Econom\'{\i}a y Competi\-tividad of Spain under grant MTM2010--15314, co-financed by the European Community fund FEDER.


\end{document}